\documentclass[ECP]{ejpecp} 

\usepackage{amsmath,amssymb,amsfonts}
\usepackage{mathtools,enumerate}
\usepackage{bbm}
\usepackage{xcolor}

\SHORTTITLE{Set-indexed and multiple sums in high dimensions}

\TITLE{Set-indexed and multiple sums in high dimensions} 

\AUTHORS{%
  Bochen~Jin\footnote{University of Bern, Switzerland.
    \EMAIL{bochen.jin@unibe.ch}}\orcid{0009-0008-3394-2973}
  \and 
  Alexander~Marynych\footnote{Kyiv School of Economics,
    Ukraine.
    \EMAIL{omarynych@kse.org.ua}}\orcid{0000-0002-6764-7330}
  \and
  Ilya~Molchanov\footnote{University of Bern, Switzerland.
    \EMAIL{ilya.molchanov@unibe.ch}}\orcid{0000-0003-2502-0334}}


\AMSSUBJ{60G50; 60G60} 
\AMSSUBJSECONDARY{62G30} 

\SUBMITTED{May 15, 2026} 
\ACCEPTED{} 

\VOLUME{0}
\YEAR{2026}
\PAPERNUM{0}
\DOI{10.1214/YY-TN}

\ABSTRACT{We consider multiple and set-indexed sums of random vectors
  taking values in Euclidean space of growing dimension. It is shown
  that, when viewed as finite metric spaces, the sets of values of
  such sums converge in probability. The limit is identified as a
  generalization of the Wiener spiral, which appears as the
  high-dimensional limit of single-index sums.}

\newcommand{\pto}{\overset{\mathbb{P}}{\to}}

\newcommand{\1}{\mathbbm{1}}

\renewcommand{\d}{{\rm d}}

\renewcommand{\epsilon}{\varepsilon}
\newcommand{\eps}{\varepsilon}
\newcommand{\R}{\mathbb{R}}

\newcommand{\N}{\mathbb{N}}

\newcommand{\E}{\mathbb{E}}
\newcommand{\ZZ}{\mathbb{Z}}
\newcommand{\MM}{\mathbb{M}}
\newcommand{\XX}{\mathbb{X}}

\newcommand{\A}{\mathcal{A}}

\newcommand{\M}{\mathcal{M}}
\newcommand{\clM}{\overline{\M}^{\,\rho_\mu}}
\newcommand{\sF}{\mathcal{F}}

\newcommand{\Prob}[1]{\mathbb{P}\left\{#1\right\}}
\renewcommand{\P}{\mathbb{P}}
\newcommand{\salg}{\mathsf{F}}
\newcommand{\mmark}{\mathsf{m}}

\newcommand{\mathscr}{\mathfrak}

\begin{document}

\allowdisplaybreaks



\section{Introduction}
\label{sec:introduction}

A \emph{Wiener spiral} $\mathbb{WS}$ is a compact metric space
isometric to a curve in $L_2([0,1])$ given by $[0,1]\ni t\mapsto \1_{[0,t]}(\cdot)\in L_2([0,1])$.
Alternatively, $\mathbb{WS}$ is isometric to a unit interval $[0,1]$
endowed with the distance $\rho_{\mathbb{WS}}(t,s)=\sqrt{|t-s|}$, $0\leq t,s\leq 1$.
It was shown in~\cite{Kabluchko+Marynych:2024} that the Wiener spiral arises naturally as a scaling limit of random walk paths with finite-second-moment increments in increasing dimensions, when these paths are viewed as compact metric subspaces of Euclidean spaces of growing dimension. The
heavy-tailed case was analyzed in~\cite{jin:mol25,KabMarRasch:2024}.  More recently, 
in~\cite{jin_2026}, these results were extended to the non-Euclidean setting of $\ell_p$ spaces.

In the present paper, we continue this line of work by considering
both multiple sums and set-indexed sums. We begin by introducing a
generalization of the Wiener spiral.  To this end, let $\XX$ be a
measurable space equipped with a $\sigma$-algebra $\salg$ and a
probability measure $\mu$. Endow $\salg$ with the pseudometric
\begin{equation}
  \label{eq:rho-mu}
  \rho_{\mu}(A,B):=\mu(A\triangle B)^{1/2},\quad A,B\in\salg.
\end{equation}
Identifying sets that differ by a $\mu$-null set turns $\rho_{\mu}$
into a metric; we continue to denote the resulting quotient space by
$(\salg,\rho_{\mu})$. An isometric realization of $(\salg,\rho_{\mu})$
is obtained by embedding indicator functions $\1_A$, $A\in\salg$,
into $L_2(\XX,\mu)$. Another realization is given by the collection
$\{\zeta(A):A\in\salg\}\subset L_2(\Omega)$, where
$(\zeta(A))_{A\in\salg}$ is a Gaussian isonormal set-indexed process
with control measure $\mu$, defined on a probability space
$(\Omega,\mathscr{F},\mathbb{P})$. Recall that this is a centered
Gaussian process with covariance structure
$\E(\zeta(A)\zeta(B))=\mu(A\cap B)$. The stated isometries follow from
\begin{multline*}
  \|\zeta(A)-\zeta(B)\|_{L_2(\Omega)}^2
  =\E(\zeta(A)-\zeta(B))^2=\mu(A)+\mu(B)-2\mu(A\cap B)\\
  =\mu(A\triangle B)=\|\1_A-\1_B\|_{L_2(\XX,\mu)}^2,\quad A,B\in \salg.
\end{multline*}

\begin{definition}
Let $\A$ be a non-empty subset of $\salg$. A Wiener spiral on $\A$ with governing measure $\mu$ is any metric space that is isometric to the subspace $(\A,\rho_{\mu})$ of $(\salg,\rho_\mu)$.
\end{definition}

Note that the classical Wiener spiral arises in this way by choosing
$\XX=[0,1]$ with $\mu$ being the Lebesgue measure and the family
$\A=\{[0,t]:t\in[0,1]\}$. Letting $\XX=[0,1]^m$ with the Borel
$\sigma$-algebra and the Lebesgue measure and choosing the family of
rectangles $\A:=\{[0,t_1]\times\cdots\times[0,t_m]\,:\,t_1,\dots,t_m\in[0,1]\}$
yields the \emph{$m$-variate Wiener spiral} denoted by
$\mathbb{WS}_m$. This space is a compact metric space, which is
isometric to a hyper-surface in $L_2([0,1]^m)$ defined by the family
of indicators
$[0,1]^m \ni (t_1,\ldots,t_m)\mapsto\1_{\prod_{i=1}^{m}[0,t_i]}(\cdot)\in L_2([0,1]^m)$.
In this setting the Gaussian isonormal process
$
  (\zeta(t_1,\ldots,t_m))_{t_1,\ldots,t_m\in
  [0,1]}=(\zeta([0,t_1]\times\cdots\times [0,t_m]))_{t_1,\ldots,t_m\in
  [0,1]}
$
is the standard $(1,m)$-Brownian sheet. If $\mu$ is a Dirac measure at
$x\in \XX$, then $(\salg,\rho_\mu)$ is isometric to a two-point
discrete metric space.

We prove that the $m$-variate Wiener spiral arises as the limit of
ranges of multiple sums, while the ranges of set-indexed sums
converge to the compact completion of a generic Wiener spiral.

The paper is organised as follows. Section~\ref{sec:setup} describes the limiting schemes that give rise to generalized Wiener spirals and states the main results. The proofs are given in Sections~\ref{sec:proof:-muptiple-sums} and~\ref{sec:proof:-set-indexed}, which treat, respectively, the cases of multiple sums and set-indexed sums. These arguments rely on the theory of multivariate martingales in the former case and on the theory of empirical processes in the latter.

\section{Setup and main results}\label{sec:setup}

Let $M$ be any measurable space (interpreted as indices or
marks). Suppose that the underlying probability space
$(\Omega,\mathscr{F},\mathbb{P})$ is rich enough to accommodate a
collection
\begin{equation}\label{eq:nu_def}
  \nu_{n}^{(d)}:=\sum_{j=1}^{N_{n}}
  \delta_{(X^{(d)}(j),\mmark^{(d)}(j))},\quad n,d\in\mathbb{N}.
\end{equation}
of finite random counting measures on $\R^d\times M$ with deterministic total
masses $N_n=\nu_{n}^{(d)}(\R^d\times M)$. Thus, $d$ represents the
ambient dimension, while $n$ controls the total mass. Throughout,
$n=n(d)$ is an arbitrary sequence with $n(d)\to\infty$ as
$d\to\infty$; hence limits as $d\to\infty$ implicitly include
$n(d)\to\infty$. We assume that $(X^{(d)}(j))_{1\leq j\leq N_n}$ are
independent random vectors distributed as $X^{(d)}$ taking values in
$\R^d$ and interpret $\mmark^{(d)}(j)$ as a mark of
$X^{(d)}(j)$, taking values in $M$. For a measurable set $A\subset M$, put
\begin{equation}\label{eq:s_def}
  S_n^{(d)}(A):=\int_{\R^d \times A}x\nu_{n}^{(d)}({\rm d}x,{\rm d}\mmark)
  =\sum_{j=1}^{N_{n}}X^{(d)}(j)\1_{\mmark^{(d)}(j)\in A}.
\end{equation}
Let $\M$ be a family of measurable subsets of $M$. Define a finite
subset of $\R^d$ by letting
\begin{equation}
  \label{eq:points_random_walk}
  \MM_n^{(d)}(\M):=\Big\{S_n^{(d)}(A)\,:\,A\in \M\Big\}.
\end{equation}
Although $\M$ may be infinite, the random set $\MM_n^{(d)}(\M)$ is
almost surely finite since $S_n^{(d)}(A)$ depends only on the subset
of indices $j \in \{1,\dots,N_n\}$ for which $\mmark^{(d)}(j)\in A$.
Hence, at most $2^{N_n}$ distinct sums can arise. The range is
therefore the image of at most $2^{N_n}$ subsets of the index set. If
the realised marks are distinct and $\M$ induces every subset of them,
then every index sub-sum is represented, although different sub-sums
may coincide as vectors. We consider $\MM_n^{(d)}(\M)$ as a
\emph{compact metric space} with the metric induced from the Euclidean
norm $\|\cdot\|_2$ on $\R^d$. The aim is to show that the suitably
normalised space $\MM_n^{(d)}(\M)$ converges in probability to a
Wiener spiral on an appropriate set $\mathcal{A}$ (with the structure
inherited from $\M$) as $d\to\infty$.

Convergence in probability of random compact metric spaces is
understood with respect to the \emph{Gromov–Hausdorff metric}. This
metric is defined as the infimum of the Hausdorff distances between
the images of two spaces under isometric maps to an arbitrary third
space, see \cite{MR1835418}.

Below we deal with two main examples of this setting.

\vspace{2mm}
\noindent
{\bf (MS):} The case of \emph{multiple sums} is obtained by putting
  $M=\ZZ^m$, for fixed $m\in\mathbb{N}$, and
  \begin{displaymath}
    \nu^{(d)}_n=\sum_{0\leq k_1,\dots,k_m\leq n}
    \delta_{(X^{(d)}(k_1,k_2,\ldots,k_m),(k_1,k_2,\ldots,k_m))}.
  \end{displaymath}
  In this case $N_n=(n+1)^m$ and we consider the family $\M$ which
  consists of integer rectangles
  $\{([0,k_1]\times \cdots \times [0,k_m])\cap \mathbb{Z}^m:
  (k_1,\ldots,k_m)\in\mathbb{N}_0^m\}$. If $m=1$,
  this corresponds to
  $ \MM_n^{(d)}(\M)=\{\hat{S}^{(d)}(0),\hat{S}^{(d)}(1),\ldots,\hat{S}^{(d)}(n)\}$,
  where $\hat{S}^{(d)}(k):=\sum_{j=0}^{k}X^{(d)}(j)$, $1\leq k\leq n$,
  is a random walk in $\R^d$ with the generic step $X^{(d)}$.

\vspace{2mm}
\noindent
{\bf (SI):} The case of \emph{set-indexed} sums is obtained by putting
  $M=\XX$ with a probability measure $\mu$ on a $\sigma$-algebra $\salg$ and
  \begin{displaymath}
    \nu^{(d)}_n=\sum_{j=1}^{n}\delta_{(X^{(d)}(j),U(j))}
  \end{displaymath}
  with $(U(j))_{j\in\N}$ being i.i.d.\ with distribution $\mu$ and
  independent of $(X^{(d)}(j))_{j\geq 1}$. In this case $\M$ is
  assumed to be non-empty and pointwise separable, in addition to being
  a VC class, as specified below.

The \emph{Vapnik–\v{C}hervonenkis (VC) class} of sets is defined as
follows; see Section~2.6.1 of \cite{MR4628026}.  The family $\M$ is
said to shatter a finite set $I=\{x_1,\ldots,x_n\}\subset M$ if each
subset $J\subset I$ can be obtained as $I\cap A$ for some $A\in\M$.
The VC dimension of $\M$ is the supremum of the cardinalities of
finite sets shattered by $\M$. The family $\M$ is said to be a
\emph{VC class} if its VC dimension is finite, that is, there exists
$n\in\mathbb{N}$ such that no subset of size $n+1$ is shattered by
$\M$.

The class $\M$ is called \emph{pointwise separable} if it contains
a countable subclass $\M_0$ such that, for every $A\in\M$, the indicator
$\1_A$ is the pointwise limit of a sequence of indicators from $\M_0$.
By~\cite[Corollary~1]{MR2995797}, a pointwise separable VC class has
finite bracketing numbers for every probability measure. Consequently,
its metric quotient is totally bounded under $\rho_\mu$, and $\clM$,
the closure of $\{\1_A:A\in\M\}$ in $L_2(\XX,\mu)$, is compact. This
closure is isometric to the metric completion of $(\M,\rho_\mu)$.

The case (MS) with $m=1$ was considered in
\cite{Kabluchko+Marynych:2024}, where it was shown that the space
$n^{-1/2}\MM_n^{(d)}(\M)$ converges to the standard Wiener spiral.
Recall the following conditions imposed in~\cite{Kabluchko+Marynych:2024}
on the distribution of the increment $X^{(d)}=(X_1^{(d)},\ldots,X_d^{(d)})$.
\begin{itemize}
\item[(a)] The increments are centered and normalised, that is,
  \begin{equation}
    \label{eq:assump_moments}
    \E X^{(d)}=0,\quad \E \|X^{(d)}\|_2^2=1.
  \end{equation}
\item[(b)] The components of $X^{(d)}$ are mutually uncorrelated, that
  is,
  \begin{equation}
    \label{eq:assump_uncor}
    \E [ X_{a}^{(d)}X_{a'}^{(d)}] =0,\quad a,a' \in \{1,\dots,
    d\},\quad  a\neq a' .
  \end{equation}
\item[(c)] The sequence $(\|X^{(d)}\|_2^2)_{d\in\N}$ is uniformly
  integrable, that is,
  \begin{equation}
    \label{eq:assump_uniform_integrability}
    \lim_{c\to\infty} \sup_{d\in \N} \E\left[\|X^{(d)}\|_2^2
      \1_{\{\|X^{(d)}\|_2^2>c\}}\right] = 0.
  \end{equation}
\item[(d)] The individual components of $X^{(d)}$ are negligible in
  the following sense:
  \begin{equation}\label{eq:assump_negligible}
    \lim_{d\to \infty}\max_{a\in\{1,\dots,d\}}\E (X_a^{(d)})^2=0.
  \end{equation}
\end{itemize}

\begin{theorem}
  \label{theo:GH_conv_to_wiener_spiral}
  Let $n=n(d)$ be an arbitrary sequence of positive integers such that
  $n(d)\to\infty$ as $d\to\infty$. Suppose that conditions (a)--(d)
  are fulfilled.
  \begin{enumerate}
  \item In the multiple sums
      (MS) setting above, with fixed $m\in\mathbb{N}$ and
    $\M$ the family of anchored integer rectangles, the random metric
    space $n^{-m/2}\MM_n^{(d)}(\M)$ converges in probability to the
    $m$-variate Wiener spiral $\mathbb{WS}_m$.
  \item In the set-indexed (SI) setting above, with i.i.d. marks of law $\mu$
    and $\M$ a non-empty pointwise separable VC class, the random metric
    space $n^{-1/2}\MM_n^{(d)}(\M)$ converges in probability to
    $\clM$, the compact completion of the Wiener spiral on $\M$
    with governing measure $\mu$.
  \end{enumerate}
\end{theorem}

Note that different normalizing factors in the two cases are due to
the fact that the (MS) setting involves $n^m$ summands, while the
(SI) setting is constructed using $n$ summands. 

\begin{remark}
  Using the canonical embedding of $\R^d$ into $\ell_2$ as the first
  $d$ coordinates, we may regard $\MM_n^{(d)}(\M)$ as a finite subset
  of $\ell_2$. Pointwise separability places $\clM$ in the
  separable closed linear span of a countable family of indicators;
  hence it can be isometrically embedded into $\ell_2$. By
  \cite[Section~2.3]{Kabluchko+Marynych:2024}, Gromov--Hausdorff
  convergence implies convergence of the corresponding
  equivalence classes up to isometries in $\ell_2$.
\end{remark}

\section{Proof: Multiple sums}\label{sec:proof:-muptiple-sums}

Recall that in the setting of multiple sums the family $\M$ consists
of integer rectangles
$([0,k_1]\times \cdots \times [0,k_m])\cap \mathbb{Z}^m$. To simplify
the notation we shall use a bold font for vectors taking values in
$\R^m$ and use a multi-index ${\bf k}=(k_1,\ldots,k_m)$ instead of a
single index $j$ in~\eqref{eq:s_def}. Thus, we define a multi-index
random walk $(S_{\bf n})_{{\bf n}\in \N_0^m}$, where $\N_0:=\N\cup\{0\}$, by
\begin{equation}
  \label{eq:random_walk_def}
  S_{\bf n}^{(d)}:=\sum_{{\bf k}\leq {\bf n}}X^{(d)}({\bf k}),\quad
  {\bf n}\in\N_0^m,
\end{equation}
where $\leq$ is understood componentwise. The range of summation
${\bf k}\leq {\bf n}$ should more clearly be written as
${\bf 0}\leq {\bf k}\leq {\bf n}$, with ${\bf 0}=(0,\ldots,0)$ but
here and in what follows we omit the lower inequality to streamline the
notation. Observe also that with this convention
$S_{\bf 0}^{(d)}=X^{(d)}({\bf 0})$ rather than $S_{\bf
  0}^{(d)}=0$. Then, the set~\eqref{eq:points_random_walk} can also
be written as
\begin{equation}
  \label{eq:points_random_walk_alt}
  \MM_n^{(d)}(\M):=\Big\{S_{\bf k}^{(d)}\,:\,{\bf k}\leq n{\bf 1}\Big\},
\end{equation}
where ${\bf 1}=(1,1,\ldots,1)$. We shall first prove the following. 

\begin{theorem}\label{thm:theo2}
  Under assumption (a)-(d) the following holds true
  \begin{equation}\label{eq:2}
    \sup_{{\bf u}\in [0,1]^m}
    \Big|n^{-m}\|S_{\lfloor n{\bf u}\rfloor}^{(d)}\|_2^2
      -u_1 u_2 \cdots u_m \Big|~\pto~0\quad \text{as}\;d\to\infty.
  \end{equation}
\end{theorem}

For the proof of Theorem~\ref{thm:theo2} observe that, for every
${\bf k}\in\mathbb{N}_0^m$,
\begin{displaymath}
\|S_{{\bf k}}^{(d)}\|_2^2=\sum_{{\bf i}\leq {\bf k}}\|X^{(d)}({\bf i})\|_2^2+Q^{(d)}({\bf k}),
\end{displaymath}
where
\begin{displaymath}
Q^{(d)}({\bf k}):=\sum_{{\bf i},{\bf j}\leq {\bf k},{\bf i}\neq {\bf
    j}}
\langle X^{(d)}({\bf i}),X^{(d)}({\bf j})\rangle,\quad {\bf k}\in\mathbb{N}_0^m,
\end{displaymath}
and $\langle\cdot,\cdot\rangle$ denotes the standard inner product in
$\R^d$. By Lemma~A.1 in~\cite{Kabluchko+Marynych:2024}, for each
fixed ${\bf u}\in[0,1]^m$,
\begin{displaymath}
 D_d({\bf u}):=n^{-m}\sum_{{\bf i}\leq\lfloor n{\bf u}\rfloor}
 \|X^{(d)}({\bf i})\|_2^2\pto g({\bf u}):=u_1\cdots u_m.
\end{displaymath}
This convergence is uniform in probability by a multivariate version
of classical Dini's theorem. Indeed, given $\eta>0$, choose a finite
coordinate grid such that the oscillation of $g$ between the lower and
upper corners of every grid cell is at most $\eta$. Pointwise
convergence at the finitely many grid points and a union bound show
that the maximum error there is at most $\eta$ with probability
tending to one. Since both $D_d$ and $g$ are coordinatewise
nondecreasing, every intermediate value is sandwiched between the two
corner values, and hence
\begin{displaymath}
 \sup_{{\bf u}\in[0,1]^m}|D_d({\bf u})-g({\bf u})|\pto0.
\end{displaymath}
Hence, to prove Theorem~\ref{thm:theo2} it suffices to verify that
\begin{equation}
  \label{eq:3}
  \sup_{{\bf u}\in [0,1]^m}\frac{|Q^{(d)}(\lfloor n{\bf  u}\rfloor)|}{n^{m}}
  ~\pto~0\quad \text{as}\; d\to\infty.
\end{equation}
To this end, define the $\sigma$-algebras
\begin{displaymath}
  \mathscr{F}^{(d)}_{s}(k):=\sigma\big(X^{(d)}({\bf i})\,:
  {\bf i}=(i_1,\ldots,i_m)\in\mathbb{N}_0^m,\,i_s\leq k\big),
  \quad 1\leq s\leq m,\quad k\geq 0,
\end{displaymath}
and the filtrations
\begin{displaymath}
  \mathscr{F}^{(d)}_{s}:=\big\{
  \mathscr{F}^{(d)}_{s}(k),\;k\geq 0\big\},\quad 1\leq s\leq m.
\end{displaymath}
We shall show that the multiparameter process $Q^{(d)}$ is a
martingale in an appropriate sense. A notion of orthomartingales, see
Section~I.1.2 in~\cite{Khoshnevisan}, is sufficient for our purposes, since it supplies all the necessary machinery like, for example, a version of Doob's inequality.
Observe that the independence of the random field implies that the associated coordinate conditional-expectation operators commute. Thus the filtrations introduced above satisfy the commuting-filtration hypothesis used in the orthomartingale theory of~\cite{Khoshnevisan}.

\begin{lemma}
  Under the assumption (a) the process
  $(Q^{(d)}({\bf k}))_{{\bf k}\in\N_0^m}$ is an
  orthomartingale with respect to the filtrations
  $\mathscr{F}^{(d)}_{1},\ldots,\mathscr{F}^{(d)}_{m}$.
\end{lemma}
\begin{proof}
Fix $s\in\{1,\ldots,m\}$ and ${\bf k}\in\mathbb N_0^m$ with
$k_s\geq1$. Adaptedness is immediate. Put
\begin{displaymath}
 H_s({\bf k})=\{{\bf i}\leq{\bf k}:i_s=k_s\},\qquad
 P_s({\bf k})=\{{\bf i}\leq{\bf k}:i_s<k_s\}.
\end{displaymath}
Then
\begin{displaymath}
 Q^{(d)}({\bf k})-Q^{(d)}({\bf k}-{\bf e}_s)
 =2\sum_{{\bf i}\in H_s({\bf k})}\sum_{{\bf j}\in P_s({\bf k})}
 \langle X^{(d)}({\bf i}),X^{(d)}({\bf j})\rangle
 +\sum_{\substack{{\bf i},{\bf j}\in H_s({\bf k})\\{\bf i}\neq{\bf j}}}
 \langle X^{(d)}({\bf i}),X^{(d)}({\bf j})\rangle.
\end{displaymath}
Conditionally on $\mathscr F_s^{(d)}(k_s-1)$, the variables indexed by
$P_s({\bf k})$ are measurable, whereas those indexed by $H_s({\bf k})$
are independent of that $\sigma$-algebra, centered, and mutually
independent. Both sums therefore have conditional mean zero. Hence
\begin{displaymath}
 \E\big[Q^{(d)}({\bf k})\mid\mathscr F_s^{(d)}(k_s-1)\big]
 =Q^{(d)}({\bf k}-{\bf e}_s),
\end{displaymath}
and the commuting-filtration property noted above proves the claim.
\end{proof}

By Lemma 2.1.1 in~\cite{Khoshnevisan},
$(|Q^{(d)}({\bf k})|)_{{\bf k}\in\mathbb{N}_0^m}$ is an
orthosubmartingale. Using Cairoli’s Strong Inequality, a version of
Doob's inequality for orthosubmartingales, see Theorem~2.3.1
in~\cite{Khoshnevisan}, and Markov's inequality we conclude that, for
every fixed $\eps>0$,
\begin{displaymath}
  \Prob{\sup_{{\bf u}\in [0,1]^m}\big|Q^{(d)}(\lfloor n{\bf u}\rfloor )\big|
    \geq n^m \eps}
  \leq \frac{1}{(\eps n^m)^2}
  \E\sup_{{\bf u}\in [0,1]^m}|Q^{(d)}
  (\lfloor n{\bf u}\rfloor )|^2
  \leq \frac{2^{2m}}{(\eps n^m)^2}\E|Q^{(d)}(n{\bf 1})|^2.
\end{displaymath}
Thus, for the proof of~\eqref{eq:3} we need to check that
\begin{equation}\label{eq:4}
  \lim_{d\to\infty} n^{-2m}\E[|Q^{(d)}(n{\bf 1})|^2]=0
\end{equation}
Put $N=(n+1)^m$ and $v_a=\E(X_a^{(d)})^2$. Expanding the
ordered-pair sum, independence and centering show that only
$({\bf k},{\bf l})=({\bf i},{\bf j})$ and
$({\bf k},{\bf l})=({\bf j},{\bf i})$ contribute; assumption~(b) then
forces the coordinate indices to agree. Consequently,
\begin{displaymath}
 \E|Q^{(d)}(n{\bf1})|^2
 =2N(N-1)\sum_{a=1}^d v_a^2
 =2\big((n+1)^{2m}-(n+1)^m\big)
   \sum_{a=1}^d\big(\E(X_a^{(d)})^2\big)^2.
\end{displaymath}
It remains to note that
\begin{equation}
  \label{eq:sum-max}
  \sum_{a=1}^{d}\E(X_a^{(d)})^2\E(X_a^{(d)})^2
  \leq \max_{a\in\{1,\ldots,d\}}\E(X_a^{(d)})^2
  \sum_{a=1}^{d}\E(X_a^{(d)})^2
  =\max_{a\in\{1,\ldots,d\}}\E(X_a^{(d)})^2
\end{equation}
converges to zero as $d\to\infty$ by assumption (d). This completes
the proof of~\eqref{eq:4} and of Theorem~\ref{thm:theo2}.

\begin{lemma}\label{lem:translated-block}
Fix $r\in\{1,\ldots,m\}$. For $0\leq q\leq n$ and
${\bf k}=(k_1,\ldots,k_m)$ with $0\leq k_s\leq n$ for $s\neq r$ and
$0\leq k_r\leq n-q$, let
\begin{displaymath}
  T^{(d)}_{q,r}({\bf k}):=\sum_{\substack{0\leq i_s\leq k_s,\ s\neq r,\\q\leq i_r\leq q+k_r}}X^{(d)}({\bf i}).
\end{displaymath}
Then, for every $\eta>0$, as $d\to\infty$,
\begin{displaymath}
   \sup_{0\leq q\leq n}
 \P\left\{
 \sup_{\substack{0\leq k_s\leq n,\ s\neq r\\
                  0\leq k_r\leq n-q}}
 \left|
 n^{-m}\bigl\|T^{(d)}_{q,r}({\bf k})\bigr\|_2^2
 -
 \prod_{s=1}^m\frac{k_s}{n}
 \right|>\eta
 \right\}
 \longrightarrow0.
\end{displaymath}
\end{lemma}
\begin{proof}
  For every $q$, translating the $r$-th coordinate makes the displayed
  field $X^{(d)}({\bf i})$ with $0\leq i_s\leq k_s$, $s\neq r$, and
  $q\leq i_r\leq q+k_r$, equal in distribution to the corresponding
  anchored field. The probability is therefore bounded by the full
  supremum in Theorem~\ref{thm:theo2}, up to the deterministic uniform
  error
\begin{displaymath}
 \sup_{0\leq k_s\leq n}\left|
   \prod_{s=1}^m\frac{k_s+1}{n}-\prod_{s=1}^m\frac{k_s}{n}\right|=O(n^{-1}).
\end{displaymath}
The bound is independent of $q$.
\end{proof}

With Theorem~\ref{thm:theo2} at hand, we are in position to prove
the {\bf (MS)} part of Theorem~\ref{theo:GH_conv_to_wiener_spiral}.  By
Corollary~7.3.28 on 
page~258 of~\cite{MR1835418}, it suffices to verify that
\begin{equation}\label{eq:1}
  \sup_{{\bf u},{\bf v}\in [0,1]^m}
  \left|\frac{\|S_{\lfloor n{\bf u}\rfloor}^{(d)}
      -S_{\lfloor n{\bf v}\rfloor}^{(d)}\|_2}{n^{m/2}}
    -\rho_{\mathbb{WS}_m}({\bf u},{\bf v})\right|~\pto~0\quad
  \text{as}\; d\to\infty,
\end{equation}
where the following is a pseudometric on the parameter cube and
induces the metric on its quotient,
\begin{displaymath}
  \rho_{\mathbb{WS}_m}({\bf u},{\bf v})
  =\big(|{\bf u}|+|{\bf v}|-2|{\bf u}\wedge {\bf v}|\big)^{1/2},
\end{displaymath}
with $|{\bf u}|=u_1\cdots u_m$ and ${\bf u\wedge v}$ denoting the
componentwise minimum. First, we prove that
\begin{equation}
  \label{eq:mult_sums_step1_proof1}
  n^{-m}\|S_{\lfloor n{\bf u}\rfloor}^{(d)}
  -S_{\lfloor n{\bf v}\rfloor}^{(d)}\|_2^2\pto
  |{\bf u}|+|{\bf v}|-2\big|{\bf u\wedge v}\big|
  \quad \text{as}\;d\to\infty
\end{equation}
for each given ${\bf u}\in [0,1]^m$ and ${\bf v}\in [0,1]^m$.  Indeed,
\begin{align*}
  \|S_{\lfloor n{\bf u}\rfloor}^{(d)}-S_{\lfloor n{\bf v}\rfloor}^{(d)}\|_2^2
  &=\Big\|\sum_{\substack{{\bf i}\leq\lfloor n{\bf u}\rfloor\\
    {\bf i}\not\leq \lfloor n{\bf v}\rfloor}}X^{(d)}({\bf i})
    -\sum_{\substack{{\bf i}\leq\lfloor n{\bf v}\rfloor\\
    {\bf i}\not\leq \lfloor n{\bf u}\rfloor}}X^{(d)}({\bf i})\Big\|_2^2\\
  &=\sum_{\substack{{\bf i}\leq\lfloor n{\bf u}\rfloor\\
    {\bf i}\not\leq \lfloor n{\bf v}\rfloor}}\|X^{(d)}({\bf i})\|_2^2
    +\sum_{\substack{{\bf i}\leq\lfloor n{\bf v}\rfloor\\
    {\bf i}\not\leq \lfloor n{\bf u}\rfloor}}\|X^{(d)}({\bf i})\|_2^2-2\sum_{\substack{{\bf i}\leq\lfloor n{\bf u}\rfloor\\
    {\bf i}\not\leq \lfloor n{\bf v}\rfloor}}
    \sum_{\substack{{\bf j}\leq\lfloor n{\bf v}\rfloor\\
    {\bf j}\not\leq \lfloor n{\bf u}\rfloor}}\langle X^{(d)}({\bf i}),
    X^{(d)}({\bf j})\rangle\\
  &\quad+\sum_{\substack{{\bf i}, {\bf j}\leq\lfloor n{\bf u}\rfloor\\
    {\bf i}, {\bf j}\not\leq \lfloor n{\bf v}\rfloor, {\bf i}\neq
    {\bf j}}}\langle X^{(d)}({\bf i}),
    X^{(d)}({\bf j})\rangle\;\;
    +\sum_{\substack{{\bf i}, {\bf j}\leq\lfloor n{\bf v}\rfloor\\
    {\bf i}, {\bf j}\not\leq \lfloor n{\bf u}\rfloor,{\bf i}\neq {\bf j}}}\langle X^{(d)}({\bf i}),
    X^{(d)}({\bf j})\rangle.
\end{align*}
Normalised by $n^{-m}$, the sum of the first two summands on the
right-hand side converges to
$|{\bf u}|+|{\bf v}|-2\big|{\bf u\wedge v}\big|$ in probability by the
weak law of large numbers, see Lemma~A.1
in~\cite{Kabluchko+Marynych:2024}. After normalisation by $n^{-m}$,
each of the last three summands converges to $0$ in
probability. Indeed,
\begin{multline*}
  \Prob{\bigg|\sum_{\substack{{\bf i}\leq\lfloor n{\bf u} \rfloor\\
    {\bf i}\not\leq \lfloor n{\bf v}\rfloor}}
    \sum_{\substack{{\bf j}\leq\lfloor n{\bf v}\rfloor\\
    {\bf j}\not\leq \lfloor n{\bf u}\rfloor}}\langle X^{(d)}({\bf i}),
    X^{(d)}({\bf j})\rangle\bigg|\geq \eps
    (n+1)^m}\\
    \leq \frac{1}{\eps^{2}(n+1)^{2m}}
    \E\bigg(\sum_{\substack{{\bf i}\leq\lfloor n{\bf u} \rfloor\\
    {\bf i}\not\leq \lfloor n{\bf v}\rfloor}}
    \sum_{\substack{{\bf j}\leq\lfloor n{\bf v} \rfloor\\
    {\bf j}\not\leq \lfloor n{\bf u}\rfloor}}\langle X^{(d)}({\bf i}),
    X^{(d)}({\bf j})\rangle\bigg)^2
    \leq \frac{1}{\eps^{2}}\sum_{a=1}^{d}\E(X_a^{(d)})^2\E(X_a^{(d)})^2,
\end{multline*}
which converges to zero by \eqref{eq:sum-max}. For each of the
last two ordered within-set sums, the same calculation has the
additional factor $2$ identified above; this does not affect the
conclusion.

Fix $l\in\N$. We approximate $S _{\lfloor n{\bf u}\rfloor}^{(d)}$ with sums taken on
the grid with step size $1/l$. Then
\begin{displaymath}
  \left|n^{-m/2}\|S_{\lfloor n{\bf u}\rfloor}^{(d)}
    -S_{\lfloor n{\bf v}\rfloor}^{(d)}\|_2
    -\rho_{\mathbb{WS}_m}({\bf u},{\bf v})\right|
  \leq I_1+I_2+I_3+I_4,
\end{displaymath}
where
\begin{displaymath}
  I_1:=n^{-m/2}\|S_{\lfloor n{\bf u}\rfloor}^{(d)}
    -S_{\lfloor n{\bf k_u}/l\rfloor}^{(d)}\|_2,\quad   I_2:=n^{-m/2}\|S_{\lfloor n{\bf v}\rfloor}^{(d)}
    -S_{\lfloor n{\bf k_v}/l\rfloor}^{(d)}\|_2
\end{displaymath}
with ${\bf k_u},{\bf k_v}\in\{0,\dots,l-1\}^m$ such that
${\bf k_u}/l\leq {\bf u}\leq ({\bf k_u}+{\bf 1})/l$ and ${\bf k_v}/l\leq {\bf v}\leq ({\bf k_v}+{\bf 1})/l$,
\begin{displaymath}
  I_3:=\left|\frac{\|S_{\lfloor n{\bf k_u}/l\rfloor}^{(d)}
    -S_{\lfloor n{\bf k_v/l}\rfloor}^{(d)}\|_2}{n^{m/2}}
  -\rho_{\mathbb{WS}_m}\left(\frac{{\bf k_u}}{l},\frac{{\bf k_v}}{l}\right)\right|,\,\,\,   I_4:=\left|\rho_{\mathbb{WS}_m}\left(\frac{{\bf k_u}}{l},\frac{{\bf k_v}}{l}\right)
    -\rho_{\mathbb{WS}_m}({\bf u},{\bf v})\right|.
\end{displaymath}
Since the family of all ${\bf k_u}$ and ${\bf k_v}$ is finite, by
\eqref{eq:mult_sums_step1_proof1} and the union bound, the supremum of
$I_3$ over ${\bf u}$, ${\bf v}$ tends to zero in
probability, as $d\to\infty$, for every fixed $l$. Furthermore,
\begin{align*}
  I_4^2&\leq \left|\rho^2_{\mathbb{WS}_m}({\bf k_u}/l,{\bf k_v}/l)
    -\rho^2_{\mathbb{WS}_m}({\bf u},{\bf v})\right|\\
  &=\big|\; |{\bf u}|+|{\bf v}|-2|{\bf u}\wedge{\bf v}|
    -|{\bf k_u}|/l^m-|{\bf k_v}|/l^m
    +2|{\bf k_u}\wedge{\bf k_v}|/l^m\big|\\
  &\leq \big|\; |{\bf u}|-|{\bf k_u}|/l^m|
  +||{\bf v}|-|{\bf k_v}|/l^m|+2||{\bf u}\wedge{\bf v}|
  -|{\bf k_u}\wedge{\bf k_v}|/l^m\big|
  \leq 4((1+l^{-1})^m-1),
\end{align*}
which can be made arbitrarily small by the choice of $l$.  
It remains to consider $I_1$; the treatment of $I_2$ is identical.
Put ${\bf a}={\bf k_u}/l$. For $r=1,\ldots,m$, let
\begin{displaymath}
 D_{r,n}({\bf u})=
 \prod_{s<r}\{0,\ldots,\lfloor nu_s\rfloor\}\times
 \{\lfloor na_r\rfloor+1,\ldots,\lfloor nu_r\rfloor\}\times
 \prod_{s>r}\{0,\ldots,\lfloor na_s\rfloor\},
\end{displaymath}
with an empty middle interval interpreted as the empty set.
These $m$ slabs are disjoint and 
\begin{displaymath}
 S^{(d)}_{\lfloor n{\bf u}\rfloor}-S^{(d)}_{\lfloor n{\bf a}\rfloor}
 =\sum_{r=1}^m\sum_{{\bf i}\in D_{r,n}({\bf u})}X^{(d)}({\bf i}).
\end{displaymath}
For each non-empty slab $D_{r,n}({\bf u})$, we apply
Lemma~\ref{lem:translated-block} with
\begin{displaymath}
 q=\lfloor na_r\rfloor+1,\qquad
 k_s=
 \begin{cases}
   \lfloor nu_s\rfloor, & s<r,\\
   \lfloor nu_r\rfloor-\lfloor na_r\rfloor-1, & s=r,\\
   \lfloor na_s\rfloor, & s>r.
 \end{cases}
\end{displaymath}
Each slab is an anchored rectangle translated only in coordinate $r$,
and, uniformly over ${\bf k_u}/l\leq{\bf u}\leq({\bf k_u}+{\bf1})/l$, the
normalised cardinality of $D_{r,n}({\bf u})$ is bounded as 
\begin{displaymath}
 n^{-m}|D_{r,n}({\bf u})|
 \leq (l^{-1}+n^{-1})(1+n^{-1})^{m-1}=l^{-1}+o(1).
\end{displaymath}
Lemma~\ref{lem:translated-block} therefore yields, for fixed $l$ and
$r$,
\begin{displaymath}
 \max_{{\bf k}\in\{0,\ldots,l-1\}^m}
 \sup_{{\bf k}/l\leq{\bf u}\leq({\bf k}+{\bf1})/l}
 n^{-m/2}\left\|\sum_{{\bf i}\in D_{r,n}({\bf u})}
 X^{(d)}({\bf i})\right\|_2\leq l^{-1/2}+o_{\P}(1).
\end{displaymath}
The triangle inequality gives $I_1\leq m l^{-1/2}+o_{\P}(1)$, and hence,
for every $\epsilon>0$,
\begin{displaymath}
 \lim_{l\to\infty}\limsup_{d\to\infty}\P\{I_1>\epsilon\}=0.
\end{displaymath}
Together with the estimates for $I_2,I_3,I_4$, this proves~\eqref{eq:1}.

\section{Proof: Set-indexed sums}
\label{sec:proof:-set-indexed}

Recall that $(X^{(d)}(i),U_i)_{i\geq 1}$ are independent copies of a
pair $(X^{(d)},U)$, where $X^{(d)}$ and $U$ are independent and $U$
has distribution $\mu$. In the following denote $Z:=(X^{(d)},U)$ and
$Z_i:=(X^{(d)}(i),U_i)$, $i \geq 1$. Recall that in the (SI) setting
$\M$ is a non-empty 
pointwise separable VC class. We prove convergence of
$n^{-1/2}\MM_n^{(d)}(\M)$ to the compact space $\clM$.

By pointwise separability, every supremum over $\M$
appearing below agrees almost surely with the corresponding supremum
over the countable pointwise-dense subclass $\M_0$. In
particular, all these suprema are measurable.

For $A\in\salg$ and $i\geq 1$, denote $Y_i^{(d)}(A):=X^{(d)}(i)\1_{U_i\in A}$, so that~\eqref{eq:s_def} takes the form
\begin{displaymath}
  S^{(d)}_n(A):=\sum_{i=1}^{n} Y^{(d)}_i(A),\quad n\in\mathbb{N}.
\end{displaymath}

\begin{lemma}\label{lm:set_1}
  In the {\bf (SI)} setting suppose that (a)--(d) hold. Then, for every
  $A\in\salg$,
  \begin{displaymath}
    n^{-1/2}\|S_n^{(d)}(A)\|_2 \pto \mu(A)^{1/2}
    \quad\text{as}\  d\to\infty.
  \end{displaymath}
\end{lemma}
\begin{proof}
  It suffices to show that $n^{-1}\|S_n^{(d)}(A)\|_2^2\pto \mu(A)$ as $d\to\infty$. Note that
  \begin{displaymath}
    n^{-1}\|S_n^{(d)}(A)\|_2^2
    =n^{-1}\sum_{i=1}^{n}\|Y_i^{(d)}(A)\|_2^2
    +n^{-1}\sum_{1\leq i\neq j\leq n}\langle Y_i^{(d)}(A),Y_j^{(d)}(A)\rangle.
  \end{displaymath}
  The first term converges to $\mu(A)$ in probability by the weak law
  of large numbers in the triangular scheme, see
  Lemma~A.1 in~\cite{Kabluchko+Marynych:2024}.
  The ordered-pair calculation gives
  \begin{displaymath}
  \E\left[\sum_{1\leq i\neq j\leq n}\langle
  Y_i^{(d)}(A),Y_j^{(d)}(A)\rangle\right]^2
  =2n(n-1)\mu(A)^2\sum_{a=1}^d\big(\E(X_a^{(d)})^2\big)^2.
  \end{displaymath}
  Hence, by Markov's inequality,
  \begin{align*}
    \Prob{\left|\sum_{1\leq i\neq j\leq n}\langle
    Y_i^{(d)}(A),Y_j^{(d)}(A)\rangle\right|\geq n\eps}
    &\leq 2\eps^{-2}\mu(A)^2
      \max_{1\leq a\leq d}\E(X_a^{(d)})^2\longrightarrow0.
  \end{align*}
  so that the second term converges to zero in probability, and the proof
  is complete.
\end{proof}

\begin{corollary}\label{cor:on-dim_conv_SI_setting}
  In the {\bf (SI)} setting suppose that (a)--(d) hold. Then, for every
  $A,B\in\salg$,
  \begin{displaymath}
    n^{-1/2}\|S_n^{(d)}(A)-S_n^{(d)}(B)\|_2 \pto \mu(A\triangle B)^{1/2}
    \quad\text{as}\  d\to\infty.
  \end{displaymath}
\end{corollary}
\begin{proof}
Write $c_{A,B}(u)=\1_A(u)-\1_B(u)$ and repeat the proof of
Lemma~\ref{lm:set_1} with $Y_i^{(d)}=X^{(d)}(i)c_{A,B}(U_i)$. Since
$c_{A,B}^2=\1_{A\triangle B}$, the diagonal term converges to
$\mu(A\triangle B)$ and the same ordered-pair estimate controls the
off-diagonal term.
\end{proof}

In order to proceed, recall that the \emph{bracketing number}
$N_{[]}(\eps,\M,\mu)$ is the least number of pairs of measurable sets
$(L_r,U_r)$ such that every $A\in\M$ belongs to a bracket
$L_r\subset A\subset U_r$ with $\mu(U_r\setminus L_r)\leq\eps$; see
\cite[Definition~2.1.6]{MR4628026}. In the (SI) setting, its
  finiteness for every $\eps>0$ follows from the pointwise
  separability and VC assumptions by~\cite[Corollary~1]{MR2995797}.

\begin{lemma}
  \label{le:unif_main}
  Suppose that (a)--(c) are satisfied and that
    $N_{[]}(\eps,\M,\mu)<\infty$, for every
    $\eps>0$. Then
  \begin{equation}
    \label{eq:cv_main}
    \sup_{A\in\M}\bigg|\frac{1}{n}\sum_{i=1}^{n}\|Y_i^{(d)}(A)\|_2^2
    -\mu(A)\bigg|\pto 0\quad\text{as} \ d\to\infty.
  \end{equation}
  In particular, in the {\bf (SI)} setting~\eqref{eq:cv_main} holds
  true under assumptions (a)--(c).
\end{lemma}
\begin{proof}
  Fix $\eps>0$ and let $(L_r,U_r)$, $1\leq r\leq N$, be a finite
  family of $\eps/4$-brackets. 
  If $L_r\subset A\subset U_r$, then
  \begin{multline*}
  \left|\frac1n\sum_{i=1}^n\|Y_i^{(d)}(A)\|_2^2-\mu(A)\right|
  \leq \frac1n\sum_{i=1}^n\|X^{(d)}(i)\|_2^2
       \1_{U_i\in U_r\setminus L_r}\\
  +\left|\frac1n\sum_{i=1}^n\|X^{(d)}(i)\|_2^2
       \1_{U_i\in L_r}-\mu(L_r)\right|+\mu(U_r\setminus L_r).
  \end{multline*}
  For every fixed measurable $C$, the triangular-array weak law of
  large numbers under (a)--(c) gives
  \begin{displaymath}
  \frac1n\sum_{i=1}^n\|X^{(d)}(i)\|_2^2\1_{U_i\in C}\pto\mu(C).
  \end{displaymath}
  Applying this at the finitely many sets $L_r$ and
  $U_r\setminus L_r$ shows that the supremum in~\eqref{eq:cv_main} is
  at most $\eps/2+o_{\P}(1)$. Since $\eps$ can be chosen arbitrarily
  small, the claim follows.
\end{proof}

\begin{lemma}
  \label{le:scale}
  In the {\bf (SI)} setting suppose that (a), (b) and (d) hold. Then
  \begin{displaymath}
    \sup_{A\in\M}\bigg|\frac{1}{n}\sum_{1\leq i\neq j\leq n}\langle
    Y_i^{(d)}(A),
    Y_j^{(d)}(A)\rangle\bigg|\pto 0\quad\text{as} \
    d\to\infty.
  \end{displaymath}
\end{lemma}
\begin{proof}
  Define functions $f_A: (\R^d\times\XX)^2\to\R$ by letting
  \begin{displaymath}
    f_A(z,z'):=\langle x\1_{u\in A}, x'\1_{u'\in A}\rangle=\langle x,
    x'\rangle\1_{(u,u')\in A\times A},
    \quad z:=(x,u),\ z':=(x',u').
  \end{displaymath}
  
  Pointwise separability of $\M$ also makes $\sF_{\M}:=\{f_A:A\in\M\}$
  pointwise measurable. Each $f_A$ is symmetric and canonical
  (degenerate of order one), since
  \begin{displaymath}
    \E(f_A(Z,Z')\mid Z)=
    \langle X\1_{U\in A},\E[X'\1_{U'\in A}]\rangle=0.
  \end{displaymath}
  Moreover, $\sF_{\M}$ is a VC-subgraph class. Indeed, writing
  $h(x,x')=\langle x,x'\rangle$, its strict subgraph is the union
  \begin{displaymath}
  \big((A\times A)\cap\{t<h(x,x')\}\big)
  \cup\big((A\times A)^c\cap\{t<0\}\big),
  \end{displaymath}
  a finite Boolean combination of a VC class and fixed measurable
  sets. Its VC characteristics depend only on the VC dimension of $\M$
  and not on $d$. Finally, define the dominating envelope
  \begin{displaymath}
    F(z,z'):=|\langle x,x'\rangle|,
    \qquad |f_A(z,z')|\leq F(z,z')\quad(A\in\M).
  \end{displaymath}
  By Markov's inequality, for all $\delta>0$,
  \begin{displaymath}
    \Prob{\sup_{A\in\M}\bigg|\frac{1}{n}\sum_{1\leq i\neq j\leq n}\langle
      Y_i^{(d)}(A), Y_j^{(d)}(A)\rangle\bigg|\geq \delta}
    \leq \delta^{-1}\frac{1}{n}\E
    \sup_{A\in\M}\Big|\sum_{1\leq i\neq j\leq n}f_A(Z_i,Z_j)\Big|.
  \end{displaymath}
  We aim to show that the right-hand side converges to zero as
  $d\to\infty$. 

  Let $Z_i$, $1\leq i\leq 2n$, be independent copies of
  $Z$. Write $T_n$ for the measure that places mass 1 on each of the
  pairs $(Z_i, Z_j)$ for $1\leq i,j\leq 2n$ with the exception of the
  $4n$ pairs for which $i=j$, $1\leq i\leq 2n$; $i=j-n$, $1\leq i\leq
  n$; $i=j+n$, $n+1\leq i\leq 2n$. The integral of $f_A^2$ with
  respect to $T_n$ is given by
  \begin{multline*}
    T_n f_A^2=\sum_{1\leq i\neq j\leq n}f_A(Z_i,Z_j)^2
    +\sum_{n+1\leq i\neq j\leq 2n}f_A(Z_i,Z_j)^2\\
    +\sum_{i=1}^{n}\sum_{\substack{j=n+1\\ j\neq i+n}}^{2n} f_A(Z_i,Z_j)^2
    +\sum_{i=n+1}^{2n}\sum_{\substack{j=1\\ j\neq i-n}}^{n} f_A(Z_i,Z_j)^2.
  \end{multline*}
  Next, define $\tau_n=(T_n F^2)^{1/2}$ and bound $\E\tau_n$ as 
  \begin{multline*}
    \E\tau_n
    \leq \bigg(\E\Big(\sum_{1\leq i\neq j\leq n}F(Z_i,Z_j)^2
      +\sum_{n+1\leq i\neq j\leq 2n}F(Z_i,Z_j)^2\\
      +\sum_{i=1}^{n}\sum_{\substack{j=n+1\\ j\neq i+n}}^{2n}F(Z_i,Z_j)^2
    +\sum_{i=n+1}^{2n}\sum_{\substack{j=1\\ j\neq i-n}}^{n}F(Z_i,Z_j)^2\Big)\bigg)^{1/2}
    \leq \Big(4n(n-1) \E F(Z,Z')^2\Big)^{1/2},
  \end{multline*}
  where $Z'=(X',U')$ is an independent copy of $Z=(X,U)$. Finally,
  \begin{align*}
    \E\tau_n
    \leq \Big(4n(n-1) \E F(Z,Z')^2\Big)^{1/2}
    &=\Big(4n(n-1)\sum_{a=1}^{d}\E(X_a^{(d)})^2\E\big((X_a^{(d)})'\big)^2\Big)^{1/2}\\
    &\leq 2n\big(\max_{1\leq a\leq d}\E(X_a^{(d)})^2\big)^{1/2}.
  \end{align*}
  Using 
    \begin{displaymath}
    f_A(Z_i,Z_j)^2=|\langle Y_i^{(d)}(A), Y_j^{(d)}(A)\rangle|^2
    \leq |\langle X^{(d)}(i), X^{(d)}(j) \rangle|^2=F(Z_i,Z_j)^2,
  \end{displaymath}
  we conclude that
  \begin{displaymath}
    \theta_n:=\frac{1}{4}\sup_{A\in\M}\, (T_n f_A^2)^{1/2}\leq \tau_n/4.
  \end{displaymath}
  For $n\geq2$, on the event $\{\tau_n=0\}$ all kernels vanish
  on the support of $T_n$ and the required estimate is trivial. On
  $\{\tau_n>0\}$, let $N(\eps,\rho_n,\sF_{\M})$ be the $\eps$-covering
  number under the metric
  \begin{displaymath}
    \rho_n(f_A,f_B):=\big(T_n(f_A-f_B)^2/\tau_n^2\big)^{1/2}.
  \end{displaymath}
  Note that
  \begin{displaymath}
    N(\eps,\rho_n,\sF_{\M})=N(\eps(Q_nF^2)^{1/2},\rho'_n,\sF_{\M}),
  \end{displaymath}
  where $Q_n=T_n/(4n(n-1))$ is a probability measure and the
  right-hand side is the covering number of $(\sF_{\M},\rho'_n)$ with
  $\rho'_n(f_A,f_B):=\big(Q_n(f_A-f_B)^2\big)^{1/2}$. If $\M$ is a VC
  class, then Theorem~2.6.7 of \cite{MR4628026} yields that, for all
  $\eps\in(0,1)$,
  \begin{displaymath}
    N(\eps,\rho_n,\sF_{\M})\leq K_1\eps^{-K_2},
  \end{displaymath}
  where $K_1$ and $K_2$ depend only on the VC characteristics and
  are uniform in $d$. The kernels are symmetric and canonical,
  so Theorem~6 of~\cite{MR888439} applies:
  \begin{align*}
    \frac{1}{n}\E\sup_{A\in\M}\Big|\sum_{1\leq i\neq j\leq n}&f_A(Z_i,Z_j)\Big|
    \leq \frac{C}{n}\E\Big(\theta_n+\tau_n\int_{0}^{\theta_n/\tau_n}\log
      N(\eps,\rho_n,\sF_{\M})\d\eps\Big)\\
    &\leq C\big(\max_{1\leq a\leq d}
      \E(X_a^{(d)})^2\big)^{1/2}
      \Big(\frac{1}{2}+2\int_{0}^{1/4}\big(\log K_1+K_2\log\eps^{-1}\big)\d\eps\Big)\\
    &=C\big(\max_{1\leq a\leq d}\E(X_a^{(d)})^2\big)^{1/2}
      \Big(\frac{1}{2}+\frac{\log K_1}{2}+\frac{K_2}{2}(1+\log 4)\Big),
  \end{align*}
  where $C>0$ is a constant. By (d), the right-hand side converges to
  zero.  
\end{proof}

\begin{theorem}
  \label{theo:un_two}
 In the {\bf (SI)} setting suppose that (a)--(d) hold. Then 
  \begin{displaymath}
    \sup_{A\in\M}\Big|n^{-1/2}\|S_n^{(d)}(A)\|_2-\mu(A)^{1/2}\Big|
    \pto 0\quad\text{as}\ d\to\infty.
  \end{displaymath}
\end{theorem}
\begin{proof}
  It suffices to prove that
  \begin{equation}
    \label{eq:6-aux}
    \sup_{A\in\M}\Big|n^{-1}\|S_n^{(d)}(A)\|_2^2-\mu(A)\Big|
    \pto 0\quad\text{as}\ d\to\infty.
  \end{equation}
  We write
  \begin{displaymath}
    \|S_n^{(d)}(A)\|_2^2
    =\sum_{i=1}^n \|Y_i^{(d)}(A)\|_2^2
    +\sum_{1\leq i\neq j\leq n} \langle Y_i^{(d)}(A),
    Y_j^{(d)}(A)\rangle. 
  \end{displaymath}
  Then \eqref{eq:6-aux} follows from Lemma~\ref{le:unif_main} and
  Lemma~\ref{le:scale}. 
\end{proof}

We now finish the proof of part two of
Theorem~\ref{theo:GH_conv_to_wiener_spiral}. Put
\begin{displaymath}
 \Delta_d:=\sup_{A,B\in\M}\left|n^{-1/2}
 \|S_n^{(d)}(A)-S_n^{(d)}(B)\|_2-\rho_\mu(A,B)\right|.
\end{displaymath}
The closure of the natural relation
\begin{displaymath}
 \{(n^{-1/2}S_n^{(d)}(A),[A]):A\in\M\}
 \subset n^{-1/2}\MM_n^{(d)}(\M)\times\clM
\end{displaymath}
is a correspondence: the first space is finite and $\M$ is dense in
$\clM$. Its distortion is at most $\Delta_d$, and hence
$d_{\rm GH}(n^{-1/2}\MM_n^{(d)}(\M),\clM)\leq\Delta_d/2$.

Fix $\delta>0$. The VC covering bound yields a finite family
$A_1,\ldots,A_N\in\M$ such that every $A\in\M$ satisfies
$\mu(A\triangle A_k)\leq\delta$ for some $k$. The triangle inequality
gives
\begin{displaymath}
 \Delta_d\leq J_{2,d}(\delta)+2R_d(\delta)+(2\delta)^{1/2},
\end{displaymath}
where
\begin{align*}
 J_{2,d}(\delta)&:=\max_{k,l}\left|n^{-1/2}
 \|S_n^{(d)}(A_k)-S_n^{(d)}(A_l)\|_2-\rho_\mu(A_k,A_l)\right|,\\
 R_d(\delta)&:=\max_k\sup_{\mu(A\triangle A_k)\leq\delta}
 n^{-1/2}\|S_n^{(d)}(A)-S_n^{(d)}(A_k)\|_2.
\end{align*}
By Corollary~\ref{cor:on-dim_conv_SI_setting},
$J_{2,d}(\delta)\pto0$ for fixed $\delta$. Furthermore,
\begin{displaymath}
 S_n^{(d)}(A)-S_n^{(d)}(A_k)
 =S_n^{(d)}(A\setminus A_k)-S_n^{(d)}(A_k\setminus A).
\end{displaymath}
The two derived classes $\{A\setminus A_k:A\in\M\}$ and
$\{A_k\setminus A:A\in\M\}$ are again pointwise separable VC classes;
their finite-bracketing property follows by applying the original
brackets after the corresponding set difference. Theorem~\ref{theo:un_two}
therefore yields
\begin{displaymath}
 R_d(\delta)\leq2\delta^{1/2}+o_{\P}(1).
\end{displaymath}
Consequently, for every $\eta>0$, first choosing $\delta$ sufficiently
small and then letting $d\to\infty$ gives $\P\{\Delta_d>\eta\}\to0$;
equivalently,
\begin{displaymath}
 \lim_{\delta\downarrow0}\limsup_{d\to\infty}
 \P\{2R_d(\delta)+(2\delta)^{1/2}>\eta\}=0.
\end{displaymath}
This proves the asserted Gromov--Hausdorff convergence.




\providecommand{\bysame}{\leavevmode\hbox to3em{\hrulefill}\thinspace}
\providecommand{\MR}{\relax\ifhmode\unskip\space\fi MR }
\providecommand{\MRhref}[2]{%
  \href{http://www.ams.org/mathscinet-getitem?mr=#1}{#2}
}
\providecommand{\href}[2]{#2}


\begin{acks}
IM is grateful to Lutz D\"umbgen for advice concerning empirical
processes. The authors are grateful to the referees for numerous
constructive remarks.
\end{acks}


\end{document}